\documentclass[10pt]{article}
\font\smallit=cmti10
\font\smalltt=cmtt10

\usepackage{amssymb,latexsym,amsmath,epsfig,amsthm} %% Add other packages as necessary

\makeatletter

\renewcommand\section{\@startsection {section}{1}{\z@}
{-30pt \@plus -1ex \@minus -.2ex}
{2.3ex \@plus.2ex}
{\normalfont\normalsize\bfseries\boldmath}}

\renewcommand\subsection{\@startsection{subsection}{2}{\z@}
{-3.25ex\@plus -1ex \@minus -.2ex}
{1.5ex \@plus .2ex}
{\normalfont\normalsize\bfseries\boldmath}}

\renewcommand{\@seccntformat}[1]{\csname the#1\endcsname. }

\makeatother

\newtheorem{theorem}{Theorem}
\newtheorem{lemma}{Lemma}
\newtheorem{rek}{Remark}

\newtheorem{proposition}{Proposition}

\begin{document}

\begin{center}
\uppercase{\bf Sets of Cardinality 6 Are Not Sum-dominant}
\vskip 20pt
{\bf H\`ung Vi\d{\^e}t Chu}\\
{\smallit Department of Mathematics, University of Illinois at Urbana-Champaign, Illinois 61820, USA}\\
{\tt hungchu2@illinois.edu}\\ 
\end{center}
\vskip 20pt
\centerline{\smallit Received: , Revised: , Accepted: , Published: } 
\vskip 30pt

\centerline{\bf Abstract}
\noindent
Given a finite set $A\subseteq \mathbb{N}$, define the sum set
$$A+A = \{a_i+a_j\mid a_i,a_j\in A\}$$ and the difference set 
$$A-A = \{a_i-a_j\mid a_i,a_j\in A\}.$$ The set $A$ is said to be sum-dominant
if $|A+A|>|A-A|$. Hegarty used a nontrivial algorithm to find that $8$ is the smallest cardinality of a sum-dominant set. Since then, Nathanson has asked for a human-understandable proof of the result. However, due to the complexity of the interactions among numbers, it is still questionable whether such a proof can be written down in full without computers' help. In this paper, we present a computer-free proof that a sum-dominant set must have at least $7$ elements. We also answer the question raised by the author of the current paper et al about the smallest sum-dominant set of primes, in terms of its largest element. Using computers, we find that the smallest sum-dominant set of primes has $73$ as its maximum, smaller than the value found before.
\pagestyle{myheadings}
\markright{\smalltt INTEGERS: 19 (2019)\hfill}
\thispagestyle{empty}
\baselineskip=12.875pt
\vskip 30pt

\section{Introduction}
\subsection{Background}
Given a finite set $A\subseteq \mathbb{N}$, define $A+A = \{a_i + a_j\,|\, a_i, a_j\in A\}$ and $A-A = \{a_i - a_j \,|\, a_i, a_j\in A\}$. The set $A$ is said to be 
\begin{itemize}
    \item \textit{sum-dominant}, if $|A+A|>|A-A|$;
    \item \textit{balanced}, if $|A+A| = |A-A|$; and
    \item \textit{difference-dominant}, if $|A+A|<|A-A|$.
\end{itemize}
Because addition is commutative, while subtraction is not, sum-dominant sets are very rare. However, it was first proved by Martin and O'Bryant \cite{MO} that as $n\rightarrow 
\infty$, the proportion of sum-dominant subsets of $\{0,1,2,\ldots,n-1\}$ is bounded below by a positive constant (about $2\cdot 10^{-7}$), which was later improved by Zhao \cite{Zh3} to about $4\cdot 10^{-4}$. The last few years have seen an explosion of papers examining the properties of sum-dominant sets: see \cite{FP, Ma, Na2, Ru1, Ru2, Ru3} for
history and overview, \cite{He,MOS,MS,Na3,Zh1} for explicit constructions, \cite{CLMS2, MO, Zh3} for positive lower bounds for the percentage of
sum-dominant sets, \cite{ILMZ,MPR} for generalized sum-dominant sets,
and \cite{AMMS,CLMS1,CNMXZ,MV,Zh2} for extensions to other settings. 

In response to Nathanson's question of the smallest sum-dominant set \cite{Na2}, Hegarty \cite{He} used a clever algorithm to find that a sum-dominant set must have at least $8$ elements. (The computer program was reported to run for about $15$ hours.) However, a human-understandable proof of the result has not been produced because of the complexity lurking behind the interactions of numbers in addition and subtraction. Nathanson \cite{Na1, Na4} asked for a human-understandable proof of the smallest cardinality of a sum-dominant set. Hegarty, through personal communication, also said that it would be nice to have such a proof written down in full. This paper proves that a set of cardinality $6$ is not sum-dominant without the use of computers. In combination with \cite[Theorem 1]{Chu}, we have a computer-free proof that a sum-dominant set must have at least $7$ elements. 

\subsection{Notation}
We introduce some notation.
\begin{itemize}
    \item Let $A$ and $B$ be sets. We write $A\rightarrow B$ to mean the introduction of elements in $A$ to $B$. For example, $\{2\}\rightarrow \{4,9,12\}$ means that we introduce the number $2$ into the set $\{4,9,12\}$.
    \item We write $a_n + \cdots +a_m$ for some $n\le m$ to mean the sum $a_n+a_{n+1}+\cdots+a_{m-1}+a_{m}$.
    \item We use a different notation to write a set, which was first introduced by Spohn \cite{Sp}. Given a set $S = \{m_1, m_2, \ldots, m_n\}$, we arrange its elements in increasing order and find the differences between two consecutive numbers to form a sequence. Suppose that $m_1 < m_2 < \cdots < m_n$, then our sequence is $m_2 - m_1, m_3 - m_2, m_4 -
m_3, \ldots , m_n - m_{n-1}$, and we represent $S = (m_1\,|\,m_2 - m_1, m_3 - m_2, m_4 - m_3, \ldots , m_n - m_{n-1}) = (m_1|a_1,\ldots,a_{n-1})$, where $a_i = m_{i+1}-m_i$. Finally, any difference in $S-S$ must be equal to at least a sum $a_i+\cdots+a_j$ for some $1\le i\le j\le n-1$. Take $S = \{3, 2, 15, 10, 9\}$, for example. We arrange the elements in increasing order to have $2$, $3$, $9$, $10$, $15$, form a sequence by looking at the difference between two consecutive numbers: $1$, $6$, $1$, $5$, and write $S = (2\,|\,1, 6, 1, 5)$. All information about a set is preserved in this notation. 
\end{itemize}
\subsection{Main results}
\begin{theorem}\label{maintheo}
A set of cardinality $6$ is not sum-dominant. 
\end{theorem}
\begin{rek}\normalfont Combined with \cite[Theorem 1]{Chu}, Theorem \ref{maintheo} says that a sum-dominant set must have at least $7$ elements. This is one step closer to the result of Hegarty; that is, a sum-dominant set must have at least $8$ elements. \end{rek}
The interactions of $6$ numbers to form the sum set and the difference set are so complicated that we need a clever division of the problem into cases and reduce the complexity considerably. We believe that to prove this theorem, case analysis is inevitable. Therefore, the question is whether the proof can be written down in full without being too overwhelming. Our main technique is to argue for a lower bound for the number of pairs of equal positive differences from $A-A$, which confines set $A$ to certain structures. The lower bound in turn gives an upper bound for the number of distinct positive differences given by numbers in $A$. 

For simplicity of notation, we denote our set $A= (0\,|\, a_1, a_2, a_3, a_4, a_5)$ for $a_i\in \mathbb{N}_{\ge 1}$. In proving that $A$ is not sum-dominant, we split our proof into two sections considering whether $a_1 = a_2$ or $a_1 \neq a_2$. In particular, Section \ref{tools} provides tools to eliminate or simplify cases in our proof as well as restricts $A$ to certain structures; Section \ref{12} and Section \ref{1not2} consider the two cases $a_1 = a_2$ and $a_1 \neq a_2$, respectively; Section \ref{appen} proves one of our lemmas; Section \ref{prime} investigates sum-dominant sets of primes; finally, Section \ref{future} mentions some open problems for future research. 

Our next result is to find the smallest sum-dominant set of primes, in terms of its largest element. The Green-Tao theorem states that the primes contain arbitrarily long arithmetic progressions. Chu et al. \cite{CNMXZ} used this theorem to prove that there are infinitely many sum-dominant set of primes. However, sum-dominant sets of primes are expected to appear much earlier before we see a long arithmetic progression. For example, the authors found $\{19, 79, 109, 139, 229, 349, 379, 439\}$ as a sum-dominant set of primes \cite{CNMXZ}. The following theorem answers their question about the smallest sum-dominant set of primes, in terms of its largest element; equivalently, about how early in the prime sequence, we see a sum-dominant set. 
\begin{theorem}
The smallest sum-dominant set of primes, in terms of its largest element, is $\{3, 5, 7, 13, 17, 19, 23, 43, 47, 53, 59, 61, 67, 71, 73\}$. This set is also unique in the sense that there is no other sum-dominant set with $73$ as its largest element.
\end{theorem}
Lastly, we also have an observation about the minimum number of elements added to an arithmetic progression to have a sum-dominant set. 
\begin{rek}\normalfont
Let $c$ be the smallest number of elements to be added to an arithmetic progression to form a sum-dominant set. Then $3\le c\le 4$. This is due to two previous works. The author of the current paper proved that adding two arbitrary numbers into an arithmetic progression does not give a sum-dominant set \cite{Chu}. So, $3\le c$. It is also known that $A^* = \{0,2\}\cup \{3,7,11,\ldots,4k-1\}\cup\{4k,4k+2\}$ is sum-dominant \cite{Na3}. Another example is the set $\{0,1,3\}\cup \{7,8,\ldots,17\}\cup \{24\}$. Hence, $c\le 4$.
\end{rek}

\section{Important results}\label{tools}
In this section, we provide all necessary tools that help reduce the complexity of the problem considerably. 
We use the definition of a symmetric set given by Nathanson \cite{Na3}: a set $A$ is symmetric if there exists a number $a$ such that $a-A = A$. If so, we say that the set $A$ is symmetric about $a$. The following proposition was proved by Nathanson \cite{Na3}. 
\begin{proposition}
A symmetric set is balanced. 
\end{proposition}
\begin{proof}
Let $A$ be a symmetric set about $a$. We have $|A+A| = |A+(a-A)| = |a+(A-A)| = |A-A|$. Hence, $A$ is balanced. 
\end{proof}
Though symmetric sets are not sum-dominant, adding a few numbers into these sets (in a clever way) can produce sum-dominant sets. Examples of such a technique were provided by Hegarty \cite{He} and Nathanson \cite{Na3}. Note that a set of numbers from an arithmetic progression is symmetric about the sum of the maximum and the minimum of the arithmetic progression. For example, the set $E= \{3,5,7,9,11\}$ is symmetric about $14$. Next, we prove a very useful lemma that establishes an upper bound for the number of distinct positive differences in $A-A$. 
\begin{lemma}\label{atmost7}
Let $A$ be a sum-dominant set with $|A| = 6$. If there exist $m_1$, $m_2$, and $m_3\in A$ such that $m_2-m_1 = m_3 - m_2$, then $A$ has at most $7$ distinct positive differences. 
\end{lemma}
\begin{proof}
Let $x$ be the number of pairs of equal positive differences given by the interaction of numbers in $A$ when we take $A-A$. For example, in our set $E$ above, $11-7 = 9-5$. So, $(11,7)$ and $(9,5)$ form a pair of equal positive differences. We need the following two inequalities
\begin{align}
    |A+A| &\ \le \ |A|(|A|+1)/2,\label{boundsum}\\
    |A-A| &\ \le \ |A|(|A|-1)+1. \label{boundiff}
\end{align}
These inequalities are not hard to prove and were used by Hegarty \cite{He} and the author of the current paper \cite{Chu}. Inequality (\ref{boundsum}) gives $|A+A| \le 21$, while Inequality (\ref{boundiff}) gives $|A-A|\le 31$. The equality in (\ref{boundsum}) is achieved if the sum of any two numbers is distinct, and the equality in (\ref{boundiff}) is achieved if the difference between any two different numbers is distinct. Because we have $x$ pairs of equal positive differences, we have $|A-A| = 31-2x$ (taking into account equal negative differences). We find a lower bound for $|A+A|$ by using \cite[Observation 13]{Chu}. Because $m_2 - m_1 = m_3 - m_2$, we have the pair of equal positive differences: $(m_2, m_1)$ and $(m_3,m_2)$. According to \cite[Observation 13]{Chu}, this pair does not give another new pair of equal positive differences. So, the existence of this pair reduces the maximum number of differences in $A-A$ by exactly $2$ while reduces the maximum number of sums in $A+A$ by exactly $1$. The rest $x-1$ pairs reduces the maximum number of differences by $2(x-1)$ while reduces the maximum number of sums by at least $(x-1)/2$. Therefore, $|A+A|\le 21 - 1 - (x-1)/2 = 20 -(x-1)/2$. Because $A$ is sum-dominant,
\begin{align*}
   20-(x-1)/2 \ \ge\ |A+A| \ >\  |A-A| \ =\ 31- 2x.
\end{align*}
We have $x\ge 8$. Hence, $|A-A|\le 31- 2\cdot 8 = 15$. Because $0\in A-A$, the number of distinct positive differences is at most $(15-1)/2 = 7$, as desired. 
\end{proof}
\begin{rek}\normalfont
In Spohn's notation, if we write $A = (0|a_1, a_2,\ldots, a_5)$, then the existence of $m_1$, $m_2$, and $m_3$ as above is equivalent to the existence of $i$, $j$, and $k$ such that $a_i+\cdots+a_j = a_{j+1}+\cdots+a_k$. Equivalently, we have an arithmetic progression of length $3$. 
\end{rek}

\begin{lemma}\label{atmost8}
Let $A$ be a sum-dominant set with $|A| = 6$. Then $A-A$ has at most $8$ distinct positive differences. 
\end{lemma}
\begin{proof}
Let $x$ be the number of pairs of equal positive differences given by the interaction of numbers in $A$ when we take $A-A$. From the proof of Lemma \ref{atmost7}, we know that $|A+A|\le 21$ and $|A-A| \le 31$. By \cite[Observation 13]{Chu}, we know that $|A-A| = 31 - 2x$, while $|A+A|\le 21- x/2$. Because
$$21-x/2\ \ge\ |A+A| \ >\ |A-A| \ =\ 31-2x,$$
we have $x\ge 7$. Hence, $|A-A|\le 31 - 2\cdot 7 = 17$. Because $0\in A-A$, the number of distinct positive differences is at most $(17-1)/2 = 8$, as desired. 
\end{proof}
\begin{rek}\normalfont
If we have numbers that form an arithmetic progression of length $3$, the upper bound for the number of distinct positive differences is reduced by $1$ (from $8$ to $7$). This is a big advantage in reducing the number of cases as we will utilize this fact later. 
\end{rek}
The following proposition will also be used intensively. 
\begin{proposition}\label{arith4}
Let $|A| = 6$ and $A$ contains an arithmetic progression of length $4$, then $A$ is not sum-dominant. 
\end{proposition}
The proof follows immediately from \cite[Theorem 2]{Chu}.
Finally, we present $15$ sets that are not sum-dominant. Most of our cases are reduced to one of these forms. 
\begin{lemma}\label{12Forms}
Let $d$, $a$, and $b$ be positive real numbers. The following sets are not sum-dominant: $S_1  = (0\,|\,d,d,2d,a,b),  \mbox{ with } a+b = d; S_2  = (0\, |\,d, d, 2d, d, a);
S_3  = (0\, |\, d, d, 2d, a, d);$
$S_4  = (0\,|\,2d,d,d,a,2d);
S_5  = (0\, |\, a,b,b,a,a); 
S_6  = (0\, |\, a+b, a, a, b, a+b);
S_7  = (0\, |\, a+b, a, a, b, a);
S_8  = (0\, |\, a,2a,a,a,b);
S_9  = (0\, |\, a+b, a, a+b, a, b);
S_{10}  = (0\, |\, a+b, 2a+b, a+b, a, b);
S_{11}  = (0\, |\, a,b,a,a+b,a);
S_{12}  = (0\, |\, a,b, a+b,a,a);
S_{13}  = (0\, |\, 2a+b, a, a, b, a);
S_{14}  = (0\, |\, a+b, a, a, b, 2a);
S_{15} = (0\,|\, a, a+b, a, b, a).$
\end{lemma}
Note that we use $0$ as the minimum element, but the minimum can be any number since sum-dominance is preserved under affine transformations. Because the proof is tedious and is not the main focus of this paper, we move the proof to Section \ref{appen}.

\section{When $a_1 = a_2$}\label{12} Because $a_1 = a_2$, Lemma \ref{atmost7} says that $|A-A|\le 7$. 
If $a_1 = a_3$, then $a_1 = a_2 = a_3$, and we have an arithmetic progression of length $4$. By Proposition \ref{arith4}, we do not have a sum-dominant set. We consider $a_1\neq a_3$. 
\begin{center}\textbf{Part I. $a_3 = a_1+a_2$}\end{center}

Our distinct positive differences include 
$$a_1 \ <\ a_1+a_2\ <\ a_2+a_3 \ <\ a_1+a_2+a_3 \ < \ a_1+\cdots+a_4 \ <\ a_1+\cdots+a_5.$$
We are allowed to have at most one more positive difference. Let $a_1 = a_2 =d$. It follows that $a_3 = 2d$. Consider two cases.

\noindent \textbf{Case I:} $a_4 = d$. We have $S_2$, which is not sum-dominant. 

\noindent \textbf{Case II:} $a_4 \neq d$. Then the difference $a_2+a_3 +a_4 = 3d+a_4$ is another positive difference, meaning that all other differences must be equal to one of the following $7$ positive differences. 
    $$d \ < \ 2d \ <\ 3d \ < \ 4d \ < \ 4d+a_4 \ <\ 4d+a_4+a_5\mbox{, } 3d+a_4.$$
Indeed, $3d+a_4$ is a new difference  because $3d<3d+a_4<4d+a_4$ but $3d+a_4\neq 4d$. Consider the difference $a_2+a_3+a_4+a_5 = 3d +a_4+ a_5$.
Either $3d + a_4+a_5 = 4d$ or $3d+a_4+a_5 = 4d+a_4$. The former gives $a_4+a_5 = d$, while the latter gives $a_5=d$. None of these produces a sum-dominant set because neither $S_1$ nor $S_3$ is sum-dominant. 
\begin{center}\textbf{Part II. $a_3 \neq a_1+a_2$}\end{center}
Because $a_2+a_3>a_1$ and $a_1\neq a_3$, we have the following list of $7$ distinct differences
\begin{align*}
&a_1 \ < \ a_1+a_2 \ < \ a_1+a_2+a_3 \ < \ a_1+\cdots+a_4 \ <\ a_1+\cdots+a_5,\\
&a_3 \ < \ a_2+a_3.
\end{align*}
Consider the difference $a_2+a_3+a_4$. Either $a_2+a_3+a_4 = a_1+a_2+a_3$ or $a_2+a_3+a_4 = a_1+a_2$. The former gives $a_1 = a_4$, while the later gives $a_1 = a_3 + a_4$. 

\noindent \textbf{Case I:} $a_1 = a_4$. Then $a_2+\cdots+a_5 = a_1+\cdots+a_4$, so $a_1 = a_5$. Because $a_1= a_2=a_4=a_5$, we have a symmetric set, which is not sum-dominant. 

\noindent \textbf{Case II:} $a_1 = a_3 + a_4$. Because $a_1 = a_2 = a_3 + a_4$, we have an arithmetic progression of length $4$. By Proposition \ref{arith4}, we do not have a sum-dominant set. 

\section{When $a_1\neq a_2$}\label{1not2}
The following are distinct positive differences

\begin{align}&a_1 \ <\ a_1+a_2 \ < \ a_1+a_2+a_3 \ < \ a_1+\cdots+a_4 \ < \ a_1+\cdots+a_5,\nonumber\\ 
&a_2.\label{whena1nota2}
\end{align}
\begin{center}
\textbf{Part I. $a_2+a_3 = a_1$} 
\end{center}
Because $a_1=a_2+a_3$, we know, by Lemma \ref{atmost7}, that the number of positive differences is at most $7$. Hence, we are allowed to have at most one more positive difference. 
We consider $a_2+a_3+a_4$. 

\noindent \textbf{Case I:} $a_2+a_3+a_4 = a_1+a_2+a_3$. So, $a_1 = a_4$. Because $a_1 = a_2+a_3 = a_4$, we have an arithmetic progression of length $4$. By Proposition \ref{arith4}, we do not have a sum-dominant set. 

\noindent \textbf{Case II:} $a_2 + a_3 + a_4 = a_1 + a_2$. So, $a_1 = a_3 + a_4$. Since $a_1 = a_2 + a_3$, we have $a_2 = a_4$. 
    \begin{enumerate}
        \item Subcase II.1: $a_2 = a_3$. Since $a_2 = a_3 = a_4$, we have an arithmetic progression of length $4$. By Proposition \ref{arith4}, we do not have a sum-dominant set.
        \item Subcase II.2: $a_2 \neq a_3$. Our $7$ distinct positive differences are
        \begin{align*}&a_1 \ < \ a_1+a_2 \ < \ a_1+a_2+a_3 \ <\ a_1+\cdots+a_4 \ <\ a_1+\cdots+a_5, \\
&a_2\mbox{, }a_3.
\end{align*}
(Note that $a_3<a_1$ since $a_1 = a_3+a_4$.) Because we cannot have a new difference besides these $7$ differences, either $a_2+\cdots+a_5 = a_1+a_2+a_3$ or $a_2+\cdots+a_5 = a_1+\cdots+a_4$ because $a_2+\cdots+a_5 = 2a_2 + a_3 + a_5 > a_1 + a_2$. 
\begin{itemize}
    \item If the former, we have $a_1 = a_4+a_5$. Because $a_1 = a_2+a_3 = a_4 + a_5$, we have an arithmetic progression of length $4$ and thus, do not have a sum-dominant set.  
    \item If the latter, we have $a_1 = a_5$. Because $a_2 = a_4$, we have a symmetric set, which is not sum-dominant. 
\end{itemize}

    \end{enumerate}
\noindent \textbf{Case III:} $a_2+a_3+a_4$ is not equal to any difference in our List (\ref{whena1nota2}). By adding $a_2+a_3+a_4$ to our list, we have $7$ distinct positive differences and this new list is exhaustive. Consider the difference $a_3$. It must be that $a_3 = a_2$. Consider $a_2+\cdots+a_5$. 
\begin{enumerate}
    \item Subcase III.1: $a_2+\cdots+a_5 = a_1+\cdots+a_4$. Equivalently, $a_1= a_5$. Let $a_2=a_3 = d$. It follows that $a_5 = a_1 = a_2+a_3 =2d$. Our set is of the form $(0|2d, d, d, a_4, 2d)$, which is $S_4$, not a sum-dominant set. 
    \item Subcase III.2: $a_2+\cdots+a_5 = a_1+a_2+a_3$. Equivalently, $a_1 = a_4+a_5$. The fact that $a_1 = a_2+a_3 = a_4+a_5$ gives us an arithmetic progression of length $4$. Hence, our set is not sum-dominant. 
    \item Subcase III.3: $a_2 +\cdots+ a_5 = a_1+a_2$. Equivalently, $a_1 = a_3+a_4+a_5$. So, $a_4 + a_5 = a_1 -a_3 =a_2$. The fact that $a_2 = a_3 = a_4+a_5$ gives us an arithmetic progression of length $4$. Hence, our set is not sum-dominant. 
\end{enumerate}
\begin{center}
    \textbf{Part II. $a_2+a_3\neq a_1$}
\end{center}

\noindent \textbf{Case I:} $a_2+a_3 = a_1+a_2$. Equivalently, $a_1= a_3$. There are two possibilities for $a_2+a_3+a_4$ because $a_1 + \cdots + a_4 > a_2 + a_3 + a_4 > a_1 +a_2$. 
\begin{enumerate}
    \item Subcase I.1: $a_2+a_3+a_4 = a_1+a_2+a_3$. Equivalently, $a_1 = a_4$. Because $a_1 = a_3$, we know that $a_3 = a_4$. We thus have at most $7$ distinct positive differences. Consider the difference $a_2+\cdots+a_5$. We know that either $a_2 + \cdots +a_5 = a_1+\cdots +a_4$ or $a_2+\cdots+a_5$ is a new difference. If the former, we have $a_1 = a_5$, which implies that $a_3 = a_4 = a_5$, giving us an arithmetic progression of length $4$. Hence, we do not have a sum-dominant set. We consider the case where $a_2+\cdots+a_5$ is a new difference. All of the positive differences are
     \begin{align*}&a_1\ < \ a_1+a_2 \ <\ a_1+a_2+a_3 \ < \ a_1+\cdots+a_4 \ < \ a_1+\cdots+a_5,\\
&a_2 \ < \ a_2+\cdots+a_5.
\end{align*}
Consider $a_3 + a_4$. The only two possible values for $a_3+a_4$ are $a_1+a_2$ and $a_2$. If the former, we have $a_1 = a_2 = a_3 = a_4$, which does not give a sum-dominant set by \cite[Lemma 8]{Chu}. If the latter, we have $a_3+a_4 = a_2$, which gives us $S_8$, not a sum-dominant set. 
    \item Subcase I.2: $a_2+a_3+a_4$ is a new difference. Our set of positive differences contains
\begin{align*}&a_1 \ <\ a_1+a_2 \ < \ a_1+a_2+a_3 \ <\ a_1+\cdots+a_4 \ < \ a_1+\cdots+a_5,\\
&a_2\ < \ a_2+a_3+a_4.
\end{align*}
We consider three possibilities for $a_2+\cdots+a_5$. 
\begin{itemize}
    \item Subcase I.1.1: $a_2 + \cdots+a_5 = a_1 + a_2+a_3$. Equivalently, $a_1 = a_4 + a_5$. Because $a_1 = a_3$, we have $a_3 = a_4 + a_5$. Hence, the above list of differences is exhaustive. Consider $a_3 + a_4$. Either $a_3 + a_4 = a_1+a_2$ or $a_3 +a_4 = a_2$. Neither of these is a sum-dominant set because neither $S_9$ nor $S_{10}$ is sum-dominant. 
    \item Subcase I.1.2: $a_2+\cdots + a_5 = a_1 + \cdots + a_4$. Equivalently, $a_1 = a_5$. Consider $a_3 + a_4$. There are four possibilities. 
    
    If $a_3 +a_4 = a_2$, we arrive at $S_{15}$.
    
    If $a_3+a_4 = a_1+a_2$, we have $a_2 = a_4$ and thus, a symmetric set, which is not sum-dominant. 
    
    If $a_3 + a_4 = a_1 + a_2 + a_3$, then $a_2+a_3 = a_4$ because $a_1 = a_3$. We arrive at $S_{11}$.
    
    If $a_3 + a_4$ is a new difference, then we have exactly $8$ distinct differences by Lemma \ref{atmost8}. Consider $a_3 + a_4 + a_5$. To have $8$ distinct differences, the only possibility is that $a_3 + a_4 + a_5 = a_1 + a_2 + a_3$; equivalently, $a_1+a_2 = a_4+a_5$. Because $a_1 = a_5$, it follows that $a_2= a_4$. So, we have a symmetric set, which is not sum-dominant.  
    
    \item Subcase I.1.3: $a_2 + \cdots + a_5$ is a new difference. Consider $a_3 + a_4$. Note that $a_3+a_4\notin \{a_2, a_1+a_2\}$ because we have $8$ distinct positive differences. Indeed, the only possibility is that $a_3+a_4 = a_1+a_2+a_3$. So, $a_1+a_2 = a_4$. However, because $a_1 = a_3$, we have $a_2+a_3 = a_4$, which contradicts that $A-A$ has $8$ positive differences. 
\end{itemize}
\end{enumerate}
\noindent \textbf{Case II:} $a_2+a_3 \neq   a_1+a_2$. Equivalently, $a_1\neq a_3$. The following are distinct positive differences
\begin{align*}&a_1 \ <\ a_1+a_2 \ <\ a_1+a_2+a_3\ <\ a_1+\cdots+a_4\ <\ a_1+\cdots+a_5,\\
&a_2\ <\ a_2+a_3.
\end{align*}

\begin{enumerate}
    \item Subcase II.1: $a_3 = a_1+a_2$. The above list contains all positive differences. It must be that $a_2+a_3+a_4 = a_1+a_2+a_3$. So, $a_1 = a_4$. We also have $a_2+\cdots+a_5 = a_1 + \cdots+a_4$. Hence, $a_1 = a_5$. We arrive at $S_{12}$, which is not sum-dominant. 
    \item Subcase II.2: $a_3 = a_2$. The above list contains all positive differences. There are three possibilities for $a_2+a_3+a_4$.
    \begin{itemize}
        \item Subcase II.2.1: $a_2+a_3+a_4 = a_1+a_2+a_3$. Equivalently, $a_1=a_4$. It follows that $a_2+\cdots+a_5 = a_1+\cdots+a_4$. So, $a_1 = a_5$. We arrive at $S_5$. 
        \item Subcase II.2.2: $a_2+a_3+a_4 = a_1+a_2$. So, $a_1 = a_3 + a_4$. The difference $a_2+\cdots+a_5$ is either equal to $a_1+a_2+a_3$ or $a_1+\cdots+a_4$. If the former, we obtain $a_1 = a_4+a_5$ and arrive at $S_7$. If the latter, we obtain $a_1=a_5$ and arrive at $S_6$. 
        \item Subcase II.2.3: $a_2+a_3+a_4 = a_1$. There are three possibilities for $a_2+\cdots+a_5$.
        
        If $a_2 + \cdots +a_5 = a_1+a_2$, then we have $a_1 = a_3 + a_4 + a_5$. We arrive at $S_{13}$.
        
        If $a_2 + \cdots + a_5 = a_1 + a_2 + a_3$, then we have $a_1 = a_4 + a_5$. We arrive at $S_{14}$.
        
        If $a_2 + \cdots + a_5 = a_1+\cdots+a_4$, then $a_1 = a_5$. So, we have $a_1 = a_2 + a_3 + a_4 = a_5$ and thus, have an arithmetic progression of length $4$. Our set is not sum-dominant.
    \end{itemize}
    \item Subcase II.3: $a_3$ is a new difference. The exhaustive list of positive differences is 
    \begin{align*}&a_1\mbox{, }a_1+a_2\mbox{, }a_1+a_2+a_3\mbox{, }a_1+\cdots+a_4\mbox{, }a_1+\cdots+a_5\mbox{, }\\
&a_2\mbox{, }a_2+a_3,\\&a_3. 
\end{align*}
There are three possibilities for $a_2+a_3+a_4$. We analyze each possibility.
\begin{itemize}
    \item Subcase II.3.1: $a_2+a_3+a_4=a_1$. Then we can have at most $7$ positive differences, which is a contradiction.  
    \item Subcase II.3.2: $a_2+a_3+a_4 = a_1+a_2$. So, $a_1 = a_3+a_4$. There are two possibilities for $a_2+\cdots+a_5$. 
    
    If $a_2+\cdots+a_5 = a_1+\cdots+a_4$, then $a_1 = a_5$, implying that $a_5 = a_3+a_4$. Then we have at most $7$ positive differences, a contradiction. 
    
    If $a_2+\cdots+a_5 = a_1+a_2+a_3$, then $a_1 = a_4+a_5$. Consider $a_4$. There are three possibilities for $a_4$. 
    If $a_4 = a_2$, then because $a_1 = a_3+a_4$, we have $a_1 = a_2+a_3$. So, we have at most $7$ positive differences, a contradiction.
    If $a_4 = a_3$ or $a_4 = a_2+a_3$, we again have at most $7$ positive differences, a contradiction. 
    \item Subcase II.3.3: $a_2+a_3+a_4=a_1+a_2+a_3$. Equivalently, $a_1=a_4$. It follows that $a_2+a_3+a_4+a_5 = a_1+a_2+a_3+a_4$. Equivalently, $a_1=a_5$. So, $a_4=a_5$, implying that we have at most $7$ distinct differences, a contradiction. 
\end{itemize}
\end{enumerate}
\section{Proof of Lemma \ref{12Forms}}\label{appen}

We first prove that $S_1$ is not sum-dominant. Note that $(0\,|\,d, d, 2d)$ represents the set $K = \{0, d, 2d, 4d\}$ and $|K-K| - |K+K| = 1$. In particular, 
\begin{align*}K+K \ &=\ \{0, d, 2d, 3d, 4d, 5d, 6d, 8d\},\\
K-K\ &=\ \{0, \pm d, \pm 2d, \pm 3d, \pm 4d\}.\end{align*}
With $\{4d+a\}\rightarrow K$, we have at most $5$ new sums. However, the set of new positive differences is $\{a, a+2d, a+3d, a+4d\}$. (These are new differences because $0<a<d$.) Denote $\{4d+a\}\cup K = K_1$. Then $|K_1-K_1| - |K_1+K_1| \ge (|K-K| - |K+K|)+(2\cdot 4-5) = 4$. Finally, $\{4d+a+b\}\rightarrow K_1$ gives at least one new positive difference, which is $4d+a+b$ itself while gives at most $6$ new sums. Hence, $|S_1-S_1| - |S_1+S_1|\ge (|K_1-K_1| -|K_1+K_1|)+(2-6)\ge 4+2-6 = 0$. Hence, $S_1$ is not sum-dominant.

Next, we prove that $S_2$ is not sum-dominant. If $a=d$ or $a= 2d$, it is an easy check that $S_2$ is not sum-dominant. Because the set $(0|d, d, 2d, d)$ is not sum-dominant, it suffices to show that adding $5d+a$ to the set gives at least as many differences as sums.  We proceed by considering two cases. 
\begin{itemize}
\item Case 1: $a<d$ or $d<a<2d$. The set of new positive differences is $\{a, a+d, a+3d, a+4d, a+5d\}$, while there are at most $6$ new sums. We are done. 
\item Case 2: $a>2d$. The following are new differences $a+3d, a+4d, a+5d$ because they are all greater than $5d$. Hence, the number of new differences is at least $6$, while there are at most $6$ new sums. We are done. 
\end{itemize}
We have shown that $S_2$ is not sum-dominant. 

We prove that $S_3$ is not sum-dominant. It is easily checked that if $a=d$ or $a=2d$, we do not have a sum-dominant set. We proceed by considering three cases. 
\begin{itemize}
    \item Case 1: $a<d$. The proof follows exactly the proof that $S_1$ is not sum-dominant.
    \item Case 2: $d<a<2d$. We have $\{a+4d, a+5d\}\rightarrow K$ gives at most 11 new sums. Because $|K-K|-|K+K|=1$, it suffices to show that there are at least $5$ new positive differences. Indeed, new differences include $a+d$, $a+2d$, $a+3d$, $a+4d$, and $a+5d$. We are done. 
    \item Case 3: $a>2d$. We have $\{a+4d, a+5d\}\rightarrow K$ gives at most 11 new sums. Because $|K-K|-|K+K|=1$, it suffices to show that there are at least $5$ new positive differences. New differences include $a+2d$, $a+3d$, $a+4d$, and $a+5d$ because each of these is greater than $4d$. 
    If $a+d\neq 4d$, we have a new difference and we are done. If $a+d = 4d$, then $a=3d$. It can be checked that $S_3$ is not sum-dominant. 
\end{itemize}
Therefore, $S_3$ is not sum-dominant. 

Let $K_4 = (0|2d,d, d) = \{0,2d,3d,4d\}$. It is easy to check that $|K_4-K_4|-|K_4+K_4| = 3$. If $a = d$ or $a=2d$, it is also easily checked that $S_4$ is not sum-dominant. Because $\{4d+a, 6d+a\}\rightarrow K_4$ gives at most $11$ new sums. It suffices to show that the number of new differences is at least $8$. We consider two following cases.
\begin{itemize}
    \item Case 1: $a<d$ or $d<a<2d$. The set of new positive differences includes $a$, $a+d$, $a+2d$, $a+3d$, and $a+4d$. We are done. 
    \item Case 2: $a>2d$. The set of new positive differences includes $a+2d$, $a+3d$, $a+4d$, and $a+6d$. We are done. 
\end{itemize}
Hence, $S_4$ is not sum-dominant. 

We prove that $S_5$ is not sum-dominant. Denote $K_5 = S_5\backslash \{3a+2b\}$. The set of all possible differences in $K_5-K_5$ is $D_5 = \{a,a+b,a+2b,2a+2b,b,2b\}$. It is an easy check that if either $2a\in D_5$ or $2a+b\in D_5$, then we do not have a sum-dominant set. To illustrate, we give an example. Suppose that $2a = a+2b$. Equivalently, $a=2b$. We have $S_5 = (0\,|\,2b,b,b,2b,2b) = \{0,2b,3b,4b,6b,8b\}$. Because $\{0,2,3,4,6,8\}$ is not sum-dominant, $S_5$ is not sum-dominant. Suppose that $\{2a, 2a+b\}\cap D_5=\emptyset$. Adding $3a+2b$ to $K_5$ gives us three new positive differences $\{2a,2a+b,3a+2b\}$. Because the number of new sums is at most $6$, we know that $S_5$ is not sum-dominant. 

We prove $S_6$ is not sum-dominant. Denote $K_6 = S_6\backslash \{4a+3b\}$. The set of all possible differences in $K_6-K_6$ is $D_6 = \{a+b, 2a+b, 3a+b, 3a+2b,a,2a,b\}$. It is an easy check that if $\{a+2b, 2a+2b\}\cap D_6\neq \emptyset$, we do not have a sum-dominant set. Suppose that $\{a+2b,2a+2b\}\cap D_6 = \emptyset$. Adding $4a+3b$ to $K_6$ gives us three new positive differences $\{a+2b,2a+2b,4a+3b\}$ while at most $6$ new sums. Hence, $S_6$ is not sum-dominant. 

We prove $S_7$ is not sum-dominant. Denote $K_7 = S_7\backslash \{4a+2b\}$. Adding $4a+2b$ to $K_7$ gives us at most two possible new sums $\{7a+3b, 7a+4b\}$. Because $4a+2b$ is a new difference, $S_7$ is not sum-dominant. 

We prove that $S_8$ is not sum-dominant. Denote $K_8 = S_8\backslash \{4a+b\}$. We have $K_8 - K_8 = D_8 = \{0,a,2a,3a,4a,5a\}$. If $\{b,a+b,2a+b\}\cap D_8 \neq \emptyset$, $S_8$ is not sum-dominant. If otherwise, $4a+b\rightarrow K_8$ gives at least $3$ new positive differences while at most $6$ new sums. Hence, $S_8$ is not sum-dominant. 

The proof that $S_i$ for $9\le i\le 15$ are not sum-dominant is similar to the proof that $S_5$ is not sum-dominant. So, we omit the proof. 
\section{The smallest sum-dominant set of primes}\label{prime}
The Green-Tao theorem guarantees that there are infinitely many sum-dominant sets of primes; that is, sum-dominant sets can be constructed using long arithmetic progressions of primes. However, sum-dominant sets are expected to appear much earlier in the prime sequence. Chu et al. constructed the set $P=\{19, 79, 109, 139, 229, 349, 379, 439\}$ using the Hardy-Littlewood $k$-tuple conjecture \cite{CNMXZ}. We summarize the idea of the construction below.

An $m$-tuple $(b_1,b_2,\ldots,b_m)$ is said to be admissible if for all integers $k\ge 2$, $\{b_1,b_2,\ldots, b_m\}$ does not cover all values modulo $k$. Clearly, we only need to check for all values of $k$ between $2$ and $m$. An integer $n$ matches the tuple if $b_1+n,b_2+n,\ldots, b_m+n$ are all primes. The Hardy-Littlewood conjecture implies that every admissible $m$-tuple is matched by infinitely many integers. 

We apply this construction to $A_{8}$ and $A_{11}$ \cite{He} to find new sum-dominant sets that appear earlier in the prime sequence. In particular, 
$$12A_{8} = \{0,24,48,96,108,120,180,204,228\}$$ is an admissible $9$-tuple. A quick search shows that $$A'_8=103+12A_8  = \{103,127,151,199,211,223,283,307,331\}$$ is a set of primes. Because sum-dominance is preserved under affine transformation, $103+12A_8$ is also sum-dominant. Similarly, $$A'_{11} = 23+6A_{11} = \{23,47,59,71,89,107,137,149,173\}$$ is sum-dominant. Both $A'_{8}$ and $A'_{11}$ are smaller than the previous set $P$ in terms of the largest element. 

We can do better with computers' help. We run an algorithm to find all sum-dominant subsets of $\{3,5,\ldots,109\}$ (all primes from $3$ to $109$). We find $2725$ sets with $$\{3, 5, 7, 13, 17, 19, 23, 43, 47, 53, 59, 61, 67, 71, 73\}$$ being the uniquely smallest. We exclude $2$ from our original set of primes because if a set $S$ of primes containing $2$ is sum-dominant, then $S\backslash \{2\}$ is also sum-dominant. (The reason is that adding $2$ to a set of odd primes gives at least $7$ more differences than sums.) This reduces our running time by a half. Because all of our $2725$ sets have their sum sets be larger than their difference sets by at most $4$, adding $2$ to any of these sets does not give a sum-dominant set. 
\section{Future work}\label{future}
We end with two questions for future research.
\begin{itemize}
    \item Is there a human-understandable proof that a set of cardinality $7$ is not sum-dominant?
    Let $A$ be a set of cardinality $7$. Then $|A+A|\le 7\cdot 8/2 = 28$, while $|A-A|\le 7\cdot 6+1 = 43$. Using the same argument in the proof of Lemma \ref{atmost8}, we know that $A-A$ has at least $11$ pairs of equal positive differences. Hence, $A-A$ has at most $21$ distinct differences; equivalently, $A-A$ has at most $10$ distinct positive differences. This bound is not good enough and requires us to consider a lot more cases than when we have only $8$ distinct positive differences. Hence, it is unknown whether a human-understandable proof can be written down in full. 
    \item What is the minimum number of elements to be added to an arithmetic progression to form a sum-dominant set? 
\end{itemize}

\end{document}